\documentclass[a4paper,11pt,reqno]{amsart}
\usepackage{amsmath}
\usepackage{amstext}
\usepackage{amsbsy}
\usepackage{amsopn}
\usepackage{upref}
\usepackage{amsthm}
\usepackage{amsfonts}
\usepackage{amssymb}
\usepackage{mathrsfs}
\usepackage{times}
\usepackage{cite}
\allowdisplaybreaks
 \newtheorem{theorem}{Theorem}

\theoremstyle{definition}
 
\theoremstyle{remark}
 
\begin{document}
\title{Holomorphic Hermite functions and ellipses}
\author[H.~Chihara]{Hiroyuki Chihara}
\address{Department of Mathematics, Faculty of Education, University of the Ryukyus, Nishihara, Okinawa 903-0213, Japan}
\email{hiroyuki.chihara@gmail.com}
\thanks{Supported by the JSPS Grant-in-Aid for Scientific Research \#16K05221.}
\begin{abstract}
In 1990 van Eijndhoven and Meyers introduced systems of holomorphic Hermite functions 
and reproducing kernel Hilbert spaces associated with the systems on the complex plane. 
Moreover they studied the relationship between the family of all their Hilbert spaces 
and a class of Gelfand-Shilov functions. 
After that, their systems of holomorphic Hermite functions have been applied 
to studying quantization on the complex plane, combinatorics, and etc. 
On the other hand, 
the author recently introduced systems of holomorphic Hermite functions 
associated with ellipses on the complex plane. 
The present paper shows that their systems of holomorphic Hermite functions are 
determined by some cases of ellipses, 
and that their reproducing kernel Hilbert spaces are 
some cases of the Segal-Bargmann spaces determined by 
the Bargmann-type transforms introduced by Sj\"ostrand. 
\end{abstract}
\maketitle
\section{Introduction}
\label{section:introduction}
Let $0<s<1$. 
We denote by $\mathscr{X}_s(\mathbb{C})$ 
the set of all holomorphic functions on $\mathbb{C}$ 
satisfying an integrability condition 
$$
\lVert\varphi\rVert_s^2
=
\int_\mathbb{C}
\lvert\varphi(z)\rvert^2
\exp\left(
    -\frac{1-s^2}{2s}\lvert{z}\rvert^2
    +
    \frac{1+s^2}{4s}
    (z^2+\bar{z}^2)
    \right)
L(dz)
<
\infty, 
$$
where $L(dz)$ is the Lebesgue measure on $\mathbb{C}\simeq\mathbb{R}^2$. 
The function space $\mathscr{X}_s(\mathbb{C})$ is a Hilbert space 
equipped with an inner product
$$
(\varphi,\psi)_s
=
\int_\mathbb{C}
\varphi(z)
\overline{\psi(z)}
\exp\left(
    -\frac{1-s^2}{2s}\lvert{z}\rvert^2
    +
    \frac{1+s^2}{4s}
    (z^2+\bar{z}^2)
    \right)
L(dz),
\quad
\varphi,\psi\in\mathscr{X}_s(\mathbb{C}).
$$
The quantity $\lVert\cdot\rVert_s$ coincides with the norm induced by the inner product. 
In \cite{EM} van Eijndhoven and Meyers first introduced 
the function space $\mathscr{X}_s(\mathbb{C})$ 
and 
holomorphic Hermite functions $\psi^s_n$ defined by 
$$
\psi^s_n(z)
=
b_{nn}(s)^{-1/2}
e^{-z^2/2}
H_n(z), 
\quad
n=0,1,2,\dotsc, 
$$
$$
b_{mn}(s)
=
\frac{\pi\sqrt{s}}{1-s}
2^n
\left(\frac{1+s}{1-s}\right)^n
n!
\delta_{mn}, 
\quad
m,n=0,1,2,\dotsc. 
$$ 
$$
H_n(z)
=
(-1)^n
e^{z^2}
\left(\frac{d}{dz}\right)^n
e^{-z^2},
\quad
n=0,1,2,\dotsc. 
$$
They proved the following properties. 
\begin{theorem}[van Eijndhoven and Meyers \cite{EM}] 
\label{theorem:EM}
\quad
\begin{itemize}
\item 
$\mathscr{X}_s(\mathbb{C})$ 
is a reproducing kernel Hilbert space with a reproducing formula 
$$
\varphi(z)
=
\int_\mathbb{C}
K_s(z,\zeta)
\varphi(\zeta)
\exp\left(
    -\frac{1-s^2}{2s}\lvert{\zeta}\rvert^2
    +
    \frac{1+s^2}{4s}
    (\zeta^2+\bar{\zeta}^2)
    \right)
L(d\zeta)
$$
for $z\in\mathbb{C}$ and $\phi\in\mathscr{X}_s(\mathbb{C})$, 
where the integral kernel is given by 
$$
K_s(z,\zeta)
=
\frac{1-s^2}{2\pi s}
\exp\left(
    \frac{1-s^2}{2s}z\bar{\zeta}
    -
    \frac{1+s^2}{4s}
    (z^2+\bar{\zeta}^2)
    \right).
$$
\item 
$\{\psi^s_n\}_{n=0}^\infty$ 
is a complete orthonormal system of $\mathscr{X}_s(\mathbb{C})$. 
\item 
$\displaystyle\bigcup_{0<s<1}\mathscr{X}_s(\mathbb{C})=S^{1/2}_{1/2}$, 
where $S^{1/2}_{1/2}$ is a Gelfand-Shilov space extended on $\mathbb{C}$ 
{\rm (}See {\rm \cite{GS})}. 
\end{itemize}
\end{theorem}
After that, the function space $\mathscr{X}_s(\mathbb{C})$ 
and the system of holomorphic Hermite functions $\{\psi^s_n\}_{n=0}^\infty$ 
have been applied to studying 
quantization on $\mathbb{C}$ and related problems 
(see \cite{szafraniec, GF, AGHS}), 
combinatorics and counting (see \cite{IS}) 
and etc. 
\par
Most recently, 
the author happened to introduce holomorphic Hermite functions 
associated with ellipses on the complex plane in \cite{chihara2}. 
In order to explain this, 
we here introduce 
the standard Bargmann transform, 
the standard Segal-Bargmann space 
and their basic properties quickly.  
See \cite{chihara2, folland} for the detail.  
For any rapidly decreasing function $f(x)$ on the real line $\mathbb{R}$, 
its standard Bargmann transform $Bf(z)$ is defined by 
$$
Bf(z)
=
2^{-1/2}
\pi^{-3/4}
\int_\mathbb{R}
e^{-z^2/4+zx-x^2/2}
f(x)
dx,
\quad
z\in\mathbb{C}. 
$$
The Bargmann transform can be extended 
for all the tempered distributions on $\mathbb{R}$ 
since
$$
\lvert{e^{-z^2/4+zx-x^2/2}}\rvert=\mathcal{O}(e^{-x^2/4})
$$
for any fixed $z\in\mathbb{C}$. 
We denote the set of all square integrable functions on $\mathbb{R}$ 
by $L^2(\mathbb{R})$, 
which is a Hilbert space equipped with an inner product 
$$
(f,g)_{L^2}
=
\int_\mathbb{R}
f(x)
\overline{g(x)}
dx, 
\quad
f,g \in L^2(\mathbb{R}). 
$$
Set $\lVert{f}\rVert_{L^2}=\sqrt{(f,f)_{L^2}}$ for short. 
Moreover, 
we denote the set of all square integrable holomorphic functions on $\mathbb{C}$ 
with respect to a weighted measure $e^{-\lvert{z}\rvert^2/2}L(dz)$ 
by $\mathscr{H}_B(\mathbb{C})$, 
which is said to be the standard Segal-Bargmann space on $\mathbb{C}$ 
and is a Hilbert space equipped with an inner product 
$$
(\varphi,\psi)_{\mathscr{H}_B}
=
\int_\mathbb{C}
\varphi(z)\overline{\psi(z)}
e^{-\lvert{z}\rvert^2/2}
L(dz), 
\quad
\varphi,\psi \in \mathscr{H}_B(\mathbb{C}). 
$$
Set 
$\lVert{\varphi}\rVert_{\mathscr{H}_B}=\sqrt{(\varphi,\varphi)_{\mathscr{H}_B}}$ 
for short. 
It is well-known that the Bargmann transform $B$ is a Hilbert space isomorphism 
of $L^2(\mathbb{R})$ onto $\mathscr{H}_B(\mathbb{C})$, that is, 
$B$ is bijective and  
$(Bf,Bg)_{\mathscr{H}_B}=(f,g)_{L^2}$ for $f,g \in L^2(\mathbb{R})$. 
The inverse of $B$ is given by the adjoint of $B$, that is, 
$$
B^\ast\varphi(x)
=
2^{-1/2}
\pi^{-3/4}
\int_\mathbb{C}
e^{-\bar{z}^2/4+\bar{z}x-x^2/2}
\varphi(z)
e^{-\lvert{z}\rvert^2/2}
L(dz), 
\quad
x\in\mathbb{R}.
$$
Note that $B^\ast$ can be extended for $L^2_B(\mathbb{C})$, 
which is the set of all square integrable functions on $\mathbb{C}$ 
with respect to a weighted measure $e^{-\lvert{z}\rvert^2/2}L(dz)$. 
This is also a Hilbert space. 
The definition of its inner product and norm is 
the same as those of $\mathscr{H}_B(\mathbb{C})$. 
Then $\mathscr{H}_B(\mathbb{C})$ is a closed subspace of $L^2_B(\mathbb{C})$, 
and the projection of $L^2_B(\mathbb{C})$ onto $\mathscr{H}_B(\mathbb{C})$ 
is given by 
$$
B{\circ}B^\ast{F(z)}
=
\frac{1}{2\pi}
\int_\mathbb{C}
e^{z\bar{\zeta}/2}
F(\zeta)
e^{-\lvert\zeta\rvert^2/2}
L(d\zeta), 
\quad
F \in L^2_B(\mathbb{C}).
$$
In particular, $\varphi=B{\circ}B^\ast\varphi$ for $\varphi\in\mathscr{H}_B(\mathbb{C})$. 
\par
Suppose that $\alpha>0$ and $\beta\in\mathbb{R}$ are fixed constants satisfying 
$(\alpha,\beta)\ne(1,0)$. 
For $\rho>0$, we consider an elliptic disk of the form 
$$
E_\rho(\alpha,\beta)
=
\{
x+i\xi\in\mathbb{C} 
\ \vert \ 
(x,\xi)\in\mathbb{R}^2, 
\lvert\alpha{x}+i(\beta{x}+\xi)\rvert\leqslant\rho
\}. 
$$
Note that $\partial E_\rho(\alpha,\beta)$ is an ellipse 
whose major and minor axes join at the origin of $\mathbb{C}$. 
For $z=x+i\xi$, set 
$$
\zeta
=
\alpha{x}+i(\beta{x}+\xi)
=
\frac{\alpha+1+i\beta}{2}
z
+
\frac{\alpha-1+i\beta}{2}
\bar{z}, 
$$
$$
\Psi^{\alpha,\beta}_0(z)
=
\exp\left(\frac{\mu}{4}z^2\right),
\quad
\mu
=
\mu(\alpha,\beta)
=
\frac{1-\alpha^2-\beta^2+2i\beta}{1+\alpha^2+\beta^2}, 
$$
$$
\Psi_n^{\alpha,\beta}(z)
=
\left\{
e^{\lambda z^2/2}
\left(\frac{d}{dz}\right)^n
e^{-\lambda z^2/2}
\right\}
\Psi^{\alpha,\beta}_0(z), 
\quad
n=1,2,3,\dotsc, 
$$
$$
\lambda
=
\lambda(\alpha,\beta)
=
\frac{2\alpha^2}{(1+\alpha^2+\beta^2)(1-\alpha^2-\beta^2-2i\beta)}. 
$$
More precisely, 
$\{\Psi^{\alpha,\beta}\}_{n=1}^\infty$  
is generated by the creation operators, 
which is the adjoint of the annihilation operator for $\Psi^{\alpha,\beta}_0$. 
Then we have 
$$
\lvert\Psi^{\alpha,\beta}_0(z)\rvert^2
e^{-\lvert{z}\rvert^2/2}
=
\exp\left(-\frac{\lvert\zeta\rvert^2}{1+\alpha^2+\beta^2}\right), 
$$
which shows that the function $\Psi^{\alpha,\beta}_0$ 
and the system $\{\Psi^{\alpha,\beta}_n\}_{n=1}^\infty$ 
are concerned with the elliptic disk $E_\rho(\alpha,\beta)$ in some sense.  
One of the results in \cite{chihara2} is the following. 
\begin{theorem}[{\cite[Theorem~4.2]{chihara2}}] 
The family 
$\{\Psi^{\alpha,\beta}_n\}_{n=0}^\infty$ 
is a complete orthogonal system of $\mathscr{H}_B(\mathbb{C})$.  
\end{theorem}
One of the purposes of the present paper is to understand 
Theorem~\ref{theorem:EM} and its subsequent results 
in the framework of the standard Segal-Bargmann space $\mathscr{H}_B(\mathbb{C})$. 
\par
The standard Bargmann transform $B$, 
the standard Segal-Bargmann space $\mathscr{H}_B(\mathbb{C})$ 
and related objects are generalised. 
In fact, Sj\"ostrand constructed more general framework 
and applied it to studying microlocal analysis. 
See, e.g., \cite{sjoestrand, chihara1} and references therein. 
In what follows we recall Sj\"ostrand's theory restricted on $\mathbb{R}$ quickly. 
The generalisation of the Bargmann transform $B$ 
is given as a global Fourier integral operator of the form 
$$
Tf(z)
=
C_\phi
\int_\mathbb{R}
e^{i\phi(z,x)}
f(x)
dx,
\quad
z\in\mathbb{C}, 
$$
where $\phi(z,x)$ is a complex-valued quadratic phase function of the form 
$$
\phi(z,x)
=
\frac{a}{2}z^2
+
bzx
+
\frac{c}{2}
x^2
$$
for an appropriate function $f$ 
with assumptions $b\ne0$ and $\operatorname{Im}c>0$, 
and 
$C_\phi=2^{-1/2}\pi^{-3/4}\lvert{b}\rvert(\operatorname{Im}c)^{-1/4}$. 
We call $T$ a Bargmann-type transform. 
Note that $T$ can be also extended for all the tempered distributions on $\mathbb{R}$ 
since $\lvert{e^{i\phi(z,x)}}\rvert=\mathcal{O}(e^{-(\operatorname{Im}c)x^2/4})$. 
The case of $a=i/2$, $b=-i$ and $c=i$ corresponds to the standard Bargmann transform $B$. 
Set 
\begin{align*}
  \Phi(z)
& =
  \frac{\lvert{bz}\rvert^2}{4\operatorname{Im}c}
  -
  \frac{b^2z^2+\bar{b}^2\bar{z}^2}{8\operatorname{Im}c}
  -
  \frac{az^2-\bar{a}\bar{z}^2}{4i},
\\
  \Psi(z,\zeta)
& =
  \frac{\lvert{b}\rvert^2z\zeta}{4\operatorname{Im}c}
  -
  \frac{b^2z^2+\bar{b}^2\zeta^2}{8\operatorname{Im}c}
  -
  \frac{az^2-\bar{a}\zeta^2}{4i}.
\end{align*}
We denote the set of all square integrable holomorphic functions on $\mathbb{C}$ 
with respect to a weighted measure $e^{-2\Phi(z)}L(dz)$ 
by $\mathscr{H}_\Phi(\mathbb{C})$, 
which is a Hilbert space equipped with an inner product 
$$
(\varphi,\psi)_{\mathscr{H}_\Phi}
=
\int_\mathbb{C}
\varphi(z)\overline{\psi(z)}
e^{-2\Phi(z)}
L(dz), 
\quad
\varphi, \psi \in \mathscr{H}_\Phi(\mathbb{C}). 
$$
Set 
$\lVert\varphi\rVert_{\mathscr{H}_\Phi}=\sqrt{(\varphi,\varphi)_{\mathscr{H}_\Phi}}$ 
for short. 
Similarly we define $L^2_\Phi(\mathbb{C})$ 
which is a Hilbert space consisting of all square integrable functions 
on $\mathbb{C}$ with respect to the weighted measure $e^{-2\Phi(z)}L(dz)$. 
We denote its inner product and norm by 
$(\cdot,\cdot)_{L^2_\Phi}$ and $\lVert\cdot\rVert_{L^2_\Phi}$ respectively. 
The operator $T$ gives a Hilbert space isomorphism of 
$L^2(\mathbb{R})$ onto $\mathscr{H}_\Phi(\mathbb{C})$, that is, 
$T$ is bijective and 
$(Tf,Tg)_{\mathscr{H}_\Phi}=(f,g)_{L^2}$ 
for $f,g \in L^2(\mathbb{R})$. 
The inverse of $T$ is the formal adjoint $T^\ast$ which is concretely given by 
$$
T^\ast\varphi(x)
=
C_\phi
\int_\mathbb{C}
e^{-i\overline{\phi(z,x)}}
\varphi(z)
e^{-2\Phi(z)}
L(dz),
\quad
x\in\mathbb{R}. 
$$
Moreover, 
$\mathscr{H}_\Phi(\mathbb{C})$ is a closed subspace of $L^2_\Phi(\mathbb{C})$, 
and the projection operator is given by 
$$
T{\circ}T^\ast{F}(z)
=
C_\Phi
\int_\mathbb{C}
e^{2\Psi(z,\bar{\zeta})}
F(\zeta)
e^{-2\Phi(\zeta)}
L(d\zeta),
\quad
F \in L^2_\Phi(\mathbb{C}), 
$$
where $C_\Phi=\lvert{b}\rvert/2\pi\operatorname{Im}c$. 
In particular, 
$\varphi=T{\circ}T^\ast\varphi$ for $\varphi\in\mathscr{H}_\Phi(\mathbb{C})$. 
\par
The purpose of the present paper is to understand 
the function space $\mathscr{X}_s(\mathbb{C})$ 
and 
the family of holomorphic Hermite functions $\{\psi^s_n\}_{n=0}^\infty$, 
and the subsequent results in the framework of Sj\"ostrand's microlocal analysis. 
More precisely, in Section~\ref{section:ellipse} 
we first study the properties of  
$\mathscr{X}_s(\mathbb{C})$ and $\{\psi^s_n\}_{n=0}^\infty$ 
in the framework of the standard Bargmann transform. 
In particular, we shall understand $\{\psi^s_n\}_{n=0}^\infty$ 
from a view point of ellipses originated in \cite{chihara2}. 
Finally, in Section~\ref{section:sjoestrand} 
we study $\mathscr{X}_s(\mathbb{C})$ and $\{\psi^s_n\}_{n=0}^\infty$, 
and some of subsequent results in the framework of Sj\"ostrand. 
\section{$\mathscr{X}_s(\mathbb{C})$ and ellipses}
\label{section:ellipse}
In this section we shall understand Theorem~\ref{theorem:EM} 
from a view point of \cite[Section~4]{chihara2}. 
Our results in the present section are the following. 
\begin{theorem}
\label{theorem:bargmann} 
$\mathscr{X}_s(\mathbb{C})$ 
and 
$\{\psi^s_n\}_{n=0}^\infty$ 
are essentially determined by 
the ellipse $\partial E_\rho(\sqrt{s},0)$ 
in the framework $\mathscr{H}_B(\mathbb{C})$. 
More precisely, we have a Hilbert space isomorphism 
\begin{equation}
\mathscr{X}_s(\mathbb{C})
\ni
\varphi(z)
\quad \mapsto \quad
\sqrt{\frac{s}{1-s^2}}
\varphi\left(\frac{s}{1-s^2}z\right)
\exp\left(\frac{1}{4} \frac{1+s^2}{1-s^2}z^2\right)
\in
\mathscr{H}_B(\mathbb{C}) 
\label{equation:isomorphism31}
\end{equation}
whose inverse is given by 
\begin{equation}
\mathscr{H}_B(\mathbb{C})
\ni
\psi(z)
\quad \mapsto \quad 
\sqrt{\frac{1-s^2}{s}}
\psi\left(\frac{1-s^2}{s}z\right)
\exp\left(-\frac{1+s^2}{4s}z^2\right)
\in
\mathscr{X}_s(\mathbb{C}), 
\label{equation:isomorphism32}
\end{equation}
and for $n=0,1,2,\dotsc$, 
\begin{align}
& \psi^s_n(z)
  \times
  \exp\left(-\frac{1-s^2}{4s}\lvert{z}\rvert^2+\frac{1+s^2}{4s}z^2\right)
\nonumber
\\
  =
& \left(-\sqrt{\frac{1-s^2}{s}}\right)^n
  b_{nn}(s)^{-1/2}
  \Psi^{\sqrt{s},0}_n
  \left(\sqrt{\frac{1-s^2}{s}}z\right)
  \times
  \exp\left(-\frac{1}{4}\left\lvert\sqrt{\frac{1-s^2}{s}}z\right\rvert^2\right).
\label{equation:hermite31}
\end{align}
Moreover, the reproducing kernel $K_s(z,\zeta)$ of $\mathscr{X}_s(\mathbb{C})$ 
can be also obtained 
by the Hilbert space isomorphism and 
the reproducing formula for $\mathscr{H}_B(\mathbb{C})$. 
\end{theorem}
Recall the definition of $\partial E_\rho(\sqrt{s},0)$, that is, 
$$
\partial E_\rho(\sqrt{s},0)
=
\{
x+i\xi\in\mathbb{C}
\ \vert \ 
(x,\xi)\in\mathbb{R}^2, 
sx^2+\xi^2=\rho^2\}. 
$$
We here remark that 
$\{\partial E_\rho(\sqrt{s},0)\ \vert \ 0<s<1, \rho>0\}$ 
is the set of all ellipses whose major and minor axes are 
contained in the real and imaginary axes respectively. 
\begin{proof}[Proof of Theorem~\ref{theorem:bargmann}] 
First we shall show that 
\eqref{equation:isomorphism31} 
is a Hilbert space isomorphism of $\mathscr{X}_s(\mathbb{C})$ 
onto $\mathscr{H}_B(\mathbb{C})$ 
and its inverse is given by \eqref{equation:isomorphism32}. 
Let $\varphi(z)$ be a Lebesgue measurable function on $\mathbb{C}$. 
Set $P(z)=\varphi(z)e^{z^2/2}$, 
which corresponds to the holomorphic Hermite polynomials $H_n(z)$. 
By using an identity of the form  
$$
-
\frac{1}{2}
=
-
\frac{1+s^2}{4s}
+
\frac{(1-s)^2}{4s}, 
$$
we deduce that 
\begin{align*}
  \lVert\varphi\rVert^2_s
& =
  \int_\mathbb{C}
  \lvert\varphi(z)\rvert^2
  \exp\left(
      -
      \frac{1-s^2}{2s}\lvert{z}\rvert^2
      +
      \frac{1+s^2}{4s}
      (z^2+\bar{z}^2)
      \right)
  L(dz)
\\
& =
  \int_\mathbb{C}
  \lvert{P(z)e^{-z^2/2}}\rvert^2
  \exp\left(
      -
      \frac{1-s^2}{2s}\lvert{z}\rvert^2
      +
      \frac{1+s^2}{4s}
      (z^2+\bar{z}^2)
      \right)
  L(dz)
\\
& =
  \int_\mathbb{C}
  \left\lvert
  P(z)
  \exp\left(-\frac{1+s^2}{4s}z^2+\frac{(1-s)^2}{4s}z^2\right)
  \right\rvert^2
  \exp\left(
      -
      \frac{1-s^2}{2s}\lvert{z}\rvert^2
      +
      \frac{1+s^2}{4s}
      (z^2+\bar{z}^2)
      \right)
  L(dz)
\\
& =
  \int_\mathbb{C}
  \left\lvert
  P(z)
  \exp\left(
      \frac{(1-s)^2}{4s}z^2
      \right)
  \right\rvert^2
  \exp\left(
      -
      \frac{1-s^2}{2s}\lvert{z}\rvert^2
      \right)
  L(dz).
\end{align*}
If we change the variable by 
$$
z=\sqrt{\frac{s}{1-s^2}}\zeta, 
$$
then we have 
\begin{align}
  \lVert\varphi\rVert^2_s
& =
  \frac{s}{1-s^2}
  \int_\mathbb{C}
  \left\lvert
  P\left(\sqrt{\frac{s}{1-s^2}}\zeta\right)
  \exp\left(\frac{1}{4}\frac{1-s}{1+s}\zeta^2\right)
  \right\rvert^2
  e^{-\lvert\zeta\rvert^2/2}
  L(d\zeta)
\label{equation:isomorphism33}
\\
& =
  \frac{s}{1-s^2}
  \int_\mathbb{C}
  \left\lvert
  \varphi\left(\sqrt{\frac{s}{1-s^2}}\zeta\right)
  \exp\left(\frac{1}{2}\frac{s}{1-s^2}\zeta^2+\frac{1}{4}\frac{1-s}{1+s}\zeta^2\right)
  \right\rvert^2
  e^{-\lvert\zeta\rvert^2/2}
  L(d\zeta)
\nonumber
\\
& =
  \frac{s}{1-s^2}
  \int_\mathbb{C}
  \left\lvert
  \varphi\left(\sqrt{\frac{s}{1-s^2}}\zeta\right)
  \exp\left(\frac{1}{4}\frac{1+s^2}{1-s^2}\zeta^2\right)
  \right\rvert^2
  e^{-\lvert\zeta\rvert^2/2}
  L(d\zeta). 
\label{equation:isomorphism34}
\end{align}
If $\varphi\in\mathscr{X}_s(\mathbb{C})$, then 
$$
\varphi\left(\sqrt{\frac{s}{1-s^2}}z\right)
\exp\left(\frac{1}{4}\frac{1+s^2}{1-s^2}z^2\right)
$$
is holomorphic in $\mathbb{C}$ and belongs to $\mathscr{H}_B(\mathbb{C})$. 
Hence \eqref{equation:isomorphism34} shows that 
\eqref{equation:isomorphism31} is an injective and isometric mapping of 
$\mathscr{X}_s(\mathbb{C})$ to $\mathscr{H}_B(\mathbb{C})$.  
In the same way one can show that 
\eqref{equation:isomorphism32} is an inverse of \eqref{equation:isomorphism31}. 
Thus \eqref{equation:isomorphism31} is a Hilbert space isomorphism of 
$\mathscr{X}_s(\mathbb{C})$ onto $\mathscr{H}_B(\mathbb{C})$, 
and its inverse is given by \eqref{equation:isomorphism32}. 
\par
Next we show that 
$\mathscr{X}_s(\mathbb{C})$ 
and 
$\{\psi^s_n\}_{n=0}^\infty$ 
are essentially determined by 
the ellipse $\partial E_\rho(\sqrt{s},0)$ 
in the framework $\mathscr{H}_B(\mathbb{C})$. 
This follows from 
\eqref{equation:isomorphism33} and the correspondence \eqref{equation:hermite31}. 
For this reason, it suffices to show the correspondence \eqref{equation:hermite31}. 
If we choose $(\alpha,\beta)=(\sqrt{s},0)$, then we have for $n=0,1,2,\dotsc$, 
$$
\mu(\sqrt{s},0)
=
\frac{1-s}{1+s},
\quad
\lambda(\sqrt{s},0)
=
\frac{2s}{1-s^2}, 
$$
\begin{align*}
  \Psi^{\sqrt{s},0}_n(z)
& =
  \left\{
  \exp\left(
      \frac{s}{1-s^2}
      z^2
      \right)
  \left(\frac{d}{dz}\right)^n
  \exp\left(
      -
      \frac{s}{1-s^2}
      z^2
      \right)
  \right\}
  \exp\left(
      \frac{1}{4}
      \frac{1-s}{1+s}
      z^2
      \right)
\\
& =
  \left(-\sqrt{\frac{s}{1-s^2}}\right)^n
  H_n\left(\sqrt{\frac{s}{1-s^2}}z\right)
  \exp\left(
       \frac{(1-s)^2}{4s}
      \left(\sqrt{\frac{s}{1-s^2}}z\right)^2
      \right).
\end{align*}
By using this, we deduce that 
\begin{align*}
& \psi^s_n(z)
  \exp\left(
      -
      \frac{1-s^2}{4s}\lvert{z}\rvert^2
      +
      \frac{1+s^2}{4s}z^2
      \right)
\\
  =
& b_{nn}(s)^{-1/2}H_n(z)e^{-z^2/2}
  \exp\left(
      -
      \frac{1-s^2}{4s}\lvert{z}\rvert^2
      +
      \frac{1+s^2}{4s}z^2
      \right)
\\
  =
& b_{nn}(s)^{-1/2}H_n(z)
  \exp\left(
      -
      \frac{1-s^2}{4s}\lvert{z}\rvert^2
      +
      \frac{(1-s)^2}{4s}z^2
      \right)
\\
  =
& \left(-\sqrt{\frac{1-s^2}{s}}\right)^n
  b_{nn}(s)^{-1/2}
  \Psi^{\sqrt{s},0}_n\left(\sqrt{\frac{1-s^2}{s}}z\right)
  \exp\left(-\frac{(1-s)^2}{4s}z^2\right)
\\
  \times
& \exp\left(
      -
      \frac{1-s^2}{4s}\lvert{z}\rvert^2
      +
      \frac{(1-s)^2}{4s}z^2
      \right)
\\
  =
& \left(-\sqrt{\frac{1-s^2}{s}}\right)^n
  b_{nn}(s)^{-1/2}
  \Psi^{\sqrt{s},0}_n\left(\sqrt{\frac{1-s^2}{s}}z\right)
  \exp\left(
      -
      \frac{1-s^2}{4s}\lvert{z}\rvert^2
       \right)
\\
  =
& \left(-\sqrt{\frac{1-s^2}{s}}\right)^n
  b_{nn}(s)^{-1/2}
  \Psi^{\sqrt{s},0}_n\left(\sqrt{\frac{1-s^2}{s}}z\right)
  \exp\left(
      -
      \frac{1}{4}
      \left\lvert{\sqrt{\frac{1-s^2}{s}}z}\right\rvert^2
       \right), 
\end{align*}
which is a desired equation \eqref{equation:hermite31}. 
\par
Finally, we show that the reproducing kernel 
$K_s(z,\zeta)$ can be obtained by the reproducing formula 
$\varphi=B{\circ}B^\ast\varphi$ for $\varphi \in \mathscr{H}_B(\mathbb{C})$. 
Let $\varphi \in \mathscr{X}_s(\mathbb{C})$. Then 
$$
\varphi\left(\frac{s}{1-s^2}z\right)
\exp\left(\frac{1}{4}\frac{1+s^2}{1-s^2}z^2\right)
\in
\mathscr{H}_B(\mathbb{C}). 
$$
Substitute this into the reproducing formula for $\mathscr{H}_B(\mathbb{C})$. 
Then we have 
$$
\varphi\left(\frac{s}{1-s^2}z\right)
\exp\left(\frac{1}{4}\frac{1+s^2}{1-s^2}z^2\right)
=
\frac{1}{2\pi}
\int_\mathbb{C}
e^{z\bar{\zeta}}
\varphi\left(\frac{s}{1-s^2}\zeta\right)
\exp\left(\frac{1}{4}\frac{1+s^2}{1-s^2}\zeta^2\right)
e^{-\lvert{\zeta}\rvert^2/2}
L(d\zeta). 
$$
Hence, 
$$
\varphi\left(\frac{s}{1-s^2}z\right)
=
\frac{1}{2\pi}
\int_\mathbb{C}
\exp\left(z\bar{\zeta}-\frac{1}{4}\frac{1+s^2}{1-s^2}z^2+\frac{1}{4}\frac{1+s^2}{1-s^2}\zeta^2\right)
\varphi\left(\frac{s}{1-s^2}\zeta\right)
e^{-\lvert{\zeta}\rvert^2/2}
L(d\zeta). 
$$
By using the change of variable 
$$
z \mapsto \sqrt{\frac{s}{1-s^2}}z
$$
for $z$ and $\zeta$, we deduce that 
\begin{align*}
  \varphi(z)
& =
  \frac{1-s^2}{2\pi s}
  \int_\mathbb{C}
  \exp\left(
      \frac{1-s^2}{s}z\bar{\zeta}
      -
      \frac{1+s^2}{4s}(z^2-\zeta^2)
      \right)
  \varphi(\zeta)
  \exp\left(
      -
      \frac{1-s^2}{2s}\lvert\zeta\rvert^2
      \right)
  L(d\zeta)
\\
& =
  \frac{1-s^2}{2\pi s}
  \int_\mathbb{C}
  \exp\left(
      \frac{1-s^2}{s}z\bar{\zeta}
      -
      \frac{1+s^2}{4s}(z^2+\bar{\zeta}^2)
      \right)
\\
& \times
  \varphi(\zeta)
  \exp\left(
      -
      \frac{1-s^2}{2s}\lvert\zeta\rvert^2
      +
      \frac{1+s^2}{4s}(\zeta^2+\bar{\zeta}^2)
      \right)
  L(d\zeta)
\\
& =
  \int_\mathbb{C}
  K_s(z,\zeta)
  \varphi(\zeta)
  \exp\left(
      -
      \frac{1-s^2}{2s}\lvert\zeta\rvert^2
      +
      \frac{1+s^2}{4s}(\zeta^2+\bar{\zeta}^2)
      \right)
  L(d\zeta). 
\end{align*}
This completes the proof. 
\end{proof}
%
%
\section{$\mathscr{X}_s(\mathbb{C})$ and the Bargmann-type transforms}
\label{section:sjoestrand}
In this section we study the function space $\mathscr{X}_s(\mathbb{C})$ 
and some related topics 
from a view point of Sj\"ostrand's theory of microlocal analysis 
based on the Bargmann-type transforms. 
We can choose the phase function $\phi(z,x)$ so that 
\begin{equation}
\Phi(z)
=
\frac{1-s^2}{4s}
\lvert{z}\rvert^2
-
\frac{1+s^2}{8s}
(z^2+\bar{z}^2).  
\label{equation:Phi}
\end{equation}
Indeed, if the constants $a$, $b$ and $c$ satisfy 
\begin{equation}
\frac{1-s^2}{4s}
=
\frac{\lvert{b}\rvert^2}{4\operatorname{Im}c},
\quad
\frac{1+s^2}{4s}
=
\frac{b^2}{4\operatorname{Im}c}
+
\frac{a}{2i},  
\label{equation:abc}
\end{equation}
then \eqref{equation:Phi} holds. 
There are uncountably many choices of the triple $(a,b,c)$ 
satisfying \eqref{equation:abc}. 
For example, the choice of 
$$
a=\frac{i}{s}, 
\quad
b={\pm}i\sqrt{1-s^2}, 
\quad
c=t+is
\quad
(t\in\mathbb{R})
$$
satisfies the condition \eqref{equation:Phi}. 
Moreover, if the condition \eqref{equation:abc} is satisfied, then 
$$
\Psi(z,\zeta)
=
\frac{1-s^2}{4s}
z\zeta
-
\frac{1+s^2}{8s}
(z^2+\zeta^2), 
$$
and 
$$
C_\Phi
=
\frac{\lvert{b}\rvert^2}{2\pi\operatorname{Im}c}
=
\frac{1-s^2}{2\pi s}. 
$$
Thus we have just proved the following theorem. 
\begin{theorem}
\label{theorem:whatsixs}
If we choose $(a,b,c)$ satisfying {\rm \eqref{equation:abc}}, then 
$\mathscr{X}_s(\mathbb{C})=\mathscr{H}_\Phi(\mathbb{C})$ and 
$$
T{\circ}T^\ast F(z)
=
\int_\mathbb{C}
K_s(z,\zeta)
F(\zeta)
\exp\left(
    -\frac{1-s^2}{2s}\lvert{\zeta}\rvert^2
    +
    \frac{1+s^2}{4s}
    (\zeta^2+\bar{\zeta}^2)
    \right)
L(d\zeta),
\quad
F \in L^2_\Phi(\mathbb{C}).
$$
\end{theorem}
The reproducing kernel Hilbert space $\mathscr{X}_s(\mathbb{C})$ 
and the system of holomorphic Hermite functions $\{\psi^s_n\}_{n=0}^\infty$ 
have been applied to studying quantization on the complex plane, 
counting and combinatorics, and etc. 
In the study of quantization on $\mathbb{C}$, 
Twareque~Ali, G\'orska, Horzela and Szafraniec 
constructed a Hilbert space isomorphism of 
$\mathscr{X}_s(\mathbb{C})$ onto $\mathscr{H}_B(\mathbb{C})$ 
and its inverse concretely. 
More precisely, they constructed an integral transform of 
$\mathscr{X}_s(\mathbb{C})$ onto $\mathscr{H}_B(\mathbb{C})$. 
They obtained its integral kernel 
by using the generating function of Hermite functions. 
Their idea works well 
since a system of monomials $\{z^n/\sqrt{\pi 2^{n+1} n!}\}_{n=0}^\infty$ 
is a complete orthonormal system of $\mathscr{H}_B(\mathbb{C})$. 
See their paper \cite[Section~III]{AGHS} for the detail. 
\par
Recall that 
$B: L^2(\mathbb{R}) \rightarrow \mathscr{H}_B(\mathbb{C})$ 
and 
$T: L^2(\mathbb{R}) \rightarrow \mathscr{H}_\Phi(\mathbb{C})$ 
are Hilbert space isomorphisms. 
By using this fact, one can construct uncountably many 
Hilbert space isomorphisms of 
$\mathscr{X}_s(\mathbb{C})$ onto $\mathscr{H}_B(\mathbb{C})$. 
Then we have the following theorem. 
\begin{theorem}
\label{theorem:isomorphism100}
If we choose $(a,b,c)$ satisfying {\rm \eqref{equation:abc}}, then 
$B{\circ}T^\ast$ is a Hilbert space isomorphism of 
$\mathscr{X}_s(\mathbb{C})$ onto $\mathscr{H}_B(\mathbb{C})$.  
\end{theorem}
Finally we will see some examples of Theorem~\ref{theorem:isomorphism100}. 
\begin{theorem}
\label{theorem:examples}
\quad
\begin{itemize}
\item 
If 
$a=i/s$, $b=-i\sqrt{1-s^2}$ and $c=is$, 
then $\mathscr{X}_s(\mathbb{C})=\mathscr{H}_\Phi(\mathbb{C})$, 
and for any $\psi\in\mathscr{X}_s(\mathbb{C})$ 
\begin{align*}
  B{\circ}T^\ast\psi(z)
& =
  \int_\mathbb{C}
  G_1(z,\zeta)
  \psi(\zeta)
  \exp\left(
      -
      \frac{1-s^2}{2s}\lvert\zeta\rvert^2
      +
      \frac{1+s^2}{4s}(\zeta^2+\bar{\zeta}^2)
      \right)
  L(d\zeta), 
\\
  G_1(z,\zeta)
& =
  \frac{\sqrt{1-s}}{2^{1/2} \pi s^{1/4}}
  \exp\left(
      \sqrt{\frac{1-s}{1+s}}z\bar{\zeta}
      +
      \frac{1-s}{4(1+s)}z^2
      -
      \frac{1-s+s^2}{2s}\bar{\zeta}^2
      \right).
\end{align*}
\item 
If 
$a=is$, $b=\sqrt{1-s^2}$ and $c=is$, 
then $\mathscr{X}_s(\mathbb{C})=\mathscr{H}_\Phi(\mathbb{C})$, 
and for any $\psi\in\mathscr{X}_s(\mathbb{C})$ 
\begin{align*}
  B{\circ}T^\ast\psi(z)
& =
  \int_\mathbb{C}
  G_2(z,\zeta)
  \psi(\zeta)
  \exp\left(
      -
      \frac{1-s^2}{2s}\lvert\zeta\rvert^2
      +
      \frac{1+s^2}{4s}(\zeta^2+\bar{\zeta}^2)
      \right)
  L(d\zeta), 
\\
  G_2(z,\zeta)
& =
  \frac{\sqrt{1-s}}{2^{1/2} \pi s^{1/4}}
  \exp\left(
      -i
      \sqrt{\frac{1-s}{1+s}}z\bar{\zeta}
      +
      \frac{1-s}{4(1+s)}z^2
      -
      \frac{1}{2}\bar{\zeta}^2
      \right).
\end{align*}
\end{itemize}
\end{theorem}
\begin{proof} 
We will check only the first example. 
The second one can be proved in the same way. 
We here omit the detail. 
Suppose that $a=i/s$, $b=i\sqrt{1-s^2}$ and $c=is$. 
Then 
\begin{align*}
  T^\ast\psi(x)
& =
  \frac{\sqrt{1-s^2}}{2^{1/2}\pi^{3/4}s^{1/4}}
  \int_\mathbb{C}
  \exp\left(
      -
      \frac{1}{2s}\bar{\zeta}^2
      +
      \sqrt{1-s^2}\bar{\zeta}x
      -
      \frac{s}{2}x^2
      \right) 
\\
& \qquad \qquad \qquad 
  \times
  \psi(\zeta)
  \exp\left(
      -
      \frac{1-s^2}{2s}\lvert\zeta\rvert^2
      +
      \frac{1+s^2}{4s}(\zeta^2+\bar{\zeta}^2)
      \right)
  L(d\zeta),
  \quad
  \psi \in \mathscr{X}_s(\mathbb{C}). 
\end{align*}
We will reduce the explicit formula of $B{\circ}T^\ast\psi$. 
Applying $B$ on $T^\ast\psi$, we have 
\begin{align*}
  B{\circ}T^\ast\psi(z)
& =
  \frac{\sqrt{1-s^2}}{2\pi^{3/2}s^{1/4}}
  \int_\mathbb{C}
  \left(
  \int_\mathbb{R}
  \exp\left(
      zx
      -
      \frac{1}{2}x^2
      +
      \sqrt{1-s^2}\bar{\zeta}x
      -
      \frac{s}{2}x^2
      \right)
  dx
  \right)
\\
& \times
  \exp\left(
      -\frac{1}{4}z^2-\frac{1}{2s}\bar{\zeta}^2
      \right)
  \psi(\zeta)
  \exp\left(
      -
      \frac{1-s^2}{2s}\lvert\zeta\rvert^2
      +
      \frac{1+s^2}{4s}(\zeta^2+\bar{\zeta}^2)
      \right)
  L(d\zeta). 
\end{align*}
Since  
\begin{align*}
& zx
  -
  \frac{1}{2}x^2
  +
  \sqrt{1-s^2}\bar{\zeta}x
  -
  \frac{s}{2}x^2 
\\
  =
& -
  \frac{1+s}{2}
  \left\{
  x
  -
  \frac{1}{1+s}
  (z+\sqrt{1-s^2}\bar{\zeta})
  \right\}^2
  +
  \frac{1}{2(1+s)}
  (z+\sqrt{1-s^2}\bar{\zeta})^2, 
\end{align*}
we deduce that 
\begin{align*}
  B{\circ}T^\ast\psi(z)
& =
  \frac{\sqrt{1-s^2}}{2\pi^{3/2}s^{1/4}}
  \int_\mathbb{C}
  \left(
  \int_\mathbb{R}
  \exp\left(
      -
      \frac{1+s}{2}
      \left\{
      x
      -
      \frac{1}{1+s}
      (z+\sqrt{1-s^2}\bar{\zeta})
      \right\}^2
      \right)
  dx
  \right)
\\
& \times
  \exp\left(
      -
      \frac{1}{4}z^2-\frac{1}{2s}\bar{\zeta}^2
      -
      \frac{1}{2(1+s)}
      (z+\sqrt{1-s^2}\bar{\zeta})^2
      \right)
\\
& \times
  \psi(\zeta)
  \exp\left(
      -
      \frac{1-s^2}{2s}\lvert\zeta\rvert^2
      +
      \frac{1+s^2}{4s}(\zeta^2+\bar{\zeta}^2)
      \right)
  L(d\zeta)
\\
& =
  \frac{\sqrt{1-s^2}}{2\pi^{3/2}s^{1/4}}
  \int_\mathbb{C}
  \left(
  \int_\mathbb{R}
  \exp\left(
      -
      \frac{1+s}{2}x^2
      \right)
  dx
  \right)
\\
& \times
  \exp\left(
      \sqrt{\frac{1-s}{1+s}}z\bar{\zeta}
      +
      \frac{1-s}{4(1+s)}z^2
      -
      \frac{1-s+s^2}{2s}\bar{\zeta}^2
      \right)
\\
& \times
  \psi(\zeta)
  \exp\left(
      -
      \frac{1-s^2}{2s}\lvert\zeta\rvert^2
      +
      \frac{1+s^2}{4s}(\zeta^2+\bar{\zeta}^2)
      \right)
  L(d\zeta) 
\\
& =
  \frac{\sqrt{1-s}}{2^{1/2}\pi s^{1/4}}
  \int_\mathbb{C}
  \exp\left(
      \sqrt{\frac{1-s}{1+s}}z\bar{\zeta}
      +
      \frac{1-s}{4(1+s)}z^2
      -
      \frac{1-s+s^2}{2s}\bar{\zeta}^2
      \right)
\\
& \times
  \psi(\zeta)
  \exp\left(
      -
      \frac{1-s^2}{2s}\lvert\zeta\rvert^2
      +
      \frac{1+s^2}{4s}(\zeta^2+\bar{\zeta}^2)
      \right)
  L(d\zeta).  
\end{align*}
This completes the proof.
\end{proof}
%
%
\begin{center}
{\sc Acknowledgements} 
\end{center}
\par
The author is very grateful to Professor Franciszek Hugon Szafraniec 
for being interested in the manuscript \cite{chihara2} 
and for teaching the author recent topics on holomorphic Hermite functions kindly. 
This was a chance to write the present paper since the author knew nothing about them. 


\begin{thebibliography}{99}
\bibitem{EM} 
Van Eijndhoven SJL, Meyers JL.  
New orthogonality relations for the Hermite polynomials and related Hilbert spaces. 
J Math Anal Appl. 1990;146:89--98.

\bibitem{GS} 
Gelfand IM, Shilov GE.  
Generalized Functions Vol.2. 
New York (NY):  
Academic Press; 1968. 

\bibitem{szafraniec} 
Szafraniec FH.  
Analytic models of the quantum harmonic oscillator.  
Contemp Math. 1998;212:269--276.

\bibitem{GF} 
Gazeau JP, Szafraniec FH.  
Holomorphic Hermite polynomials and a non-commutative plane.  
J Phys A: Math Theor. 2011;44:495201 13pp.

\bibitem{AGHS} 
Twareque~Ali S, G\'orska K, Horzela A, Szafraniec FH.  
Squeezed states and Hermite polynomials in a complex variable. 
J Math Phys. 2014;55:012107 11pp.

\bibitem{IS} 
Ismail MEH, Simeonov P.  
Complex Hermite polynomials: their combinatorics and integral operators.  
Proc Amer Math Soc. 2014;143:1397--1410. 

\bibitem{chihara2} 
Chihara H.  
Bargmann-type transforms and modified harmonic oscillators.  
arXiv:1702.06646.

\bibitem{folland} 
Folland GB.  
Harmonic Analysis in Phase Space. 
Princeton (NJ):  
Princeton University Press; 1989. 

\bibitem{sjoestrand} 
Sj\"ostrand J.  
Function spaces associated to global I-Lagrangian manifolds. 
In: Morimoto M, Kawai T, editors. 
Structure of Solutions to Differential Equations, Katata/Kyoto 1995. 
River Edge (NJ): 
World Scientific Publishing,; 1996. 
p. 369--423. 

\bibitem{chihara1} 
Chihara H.  
Bounded Berezin-Toeplitz operators on the Segal-Bargmann space.  
Integral Equations Operator Theory. 2009;63:321--335.

\end{thebibliography}
\end{document}